\pdfoutput=1 

\documentclass[11pt]{amsart}
\usepackage{amsmath,amsthm,amssymb}
\usepackage{graphicx, stmaryrd}
\usepackage[usenames]{color}
\usepackage[margin=1.2in]{geometry}

\newcommand{\bu}[1]{{\mathfrak{b}_#1}}

\newcommand{\R} {{\mathbb R}}                              

\newcommand{\K} {\mathcal{K}}                           

\newcommand{\KG} {{\mathcal{K}} G}

\newcommand{\hide}[1]{}

\newcommand{\per}{\mathcal{P}}

\newcommand{\suchthat} {\:\: | \:\:}

\newcommand{\M} [1] {\ensuremath{{\overline{\mathcal M}}{_{0, #1}(\R)}}}    

\newcommand{\map}{\varphi}

\newcommand{\fil}{\textup{fill}}

\def\KP{{\mathcal K}P}
\def\K{\mathcal K}

\definecolor{mains}{cmyk}{0, .80, .80, 0}  

\theoremstyle{plain}
\newtheorem{thm}{Theorem}
\newtheorem{prop}[thm]{Proposition}
\newtheorem{cor}[thm]{Corollary}
\newtheorem{lem}[thm]{Lemma}

\theoremstyle{definition}
\newtheorem*{defn}{Definition}
\newtheorem*{exmp}{Example}

\theoremstyle{remark}
\newtheorem*{rem}{Remark}

\numberwithin{equation}{section}

\begin{document}


\title{Convex Polytopes from Nested Posets}


\author{Satyan L.\ Devadoss}
\address{S.\ Devadoss: Williams College, Williamstown, MA 01267}
\email{satyan.devadoss@williams.edu}

\author{Stefan Forcey}
\address{S.\ Forcey: University of Akron, OH 44325}
\email{sf34@uakron.edu}

\author{Stephen Reisdorf}
\address{S.\ Reisdorf: University of Akron, OH 44325}
\email{stephenreisdorf@gmail.com}

\author{Patrick Showers}
\address{P.\ Showers: University of Akron, OH 44325}
\email{pjs36@zips.uakron.edu}     
     
\begin{abstract}
Motivated by the graph associahedron $\KG$, a polytope whose face poset is based on connected subgraphs of $G$, we consider the notion of associativity and tubes on posets.  This leads to a new family of simple convex polytopes obtained by iterated truncations.  These generalize graph associahedra and nestohedra, even encompassing notions of nestings on CW-complexes.  However, these \emph{poset associahedra} fall in a different category altogether than generalized permutohedra.
\end{abstract}


\maketitle

\baselineskip=17pt

%
%
\section{Background}
\subsection{}

Given a finite graph $G$, the graph associahedron $\KG$ is a polytope whose face poset is based on the connected subgraphs of $G$ \cite{cdf}.  For special examples of graphs, $\KG$ becomes well-known, sometimes classical:  when $G$ is a path, a cycle, or a complete graph, $\KG$ results in the associahedron, cyclohedron, and permutohedron, respectively.  Figure~\ref{f:2d-tubes} shows some examples, for a graph and a pseudograph with multiple edges.
\begin{figure}[h]
\includegraphics{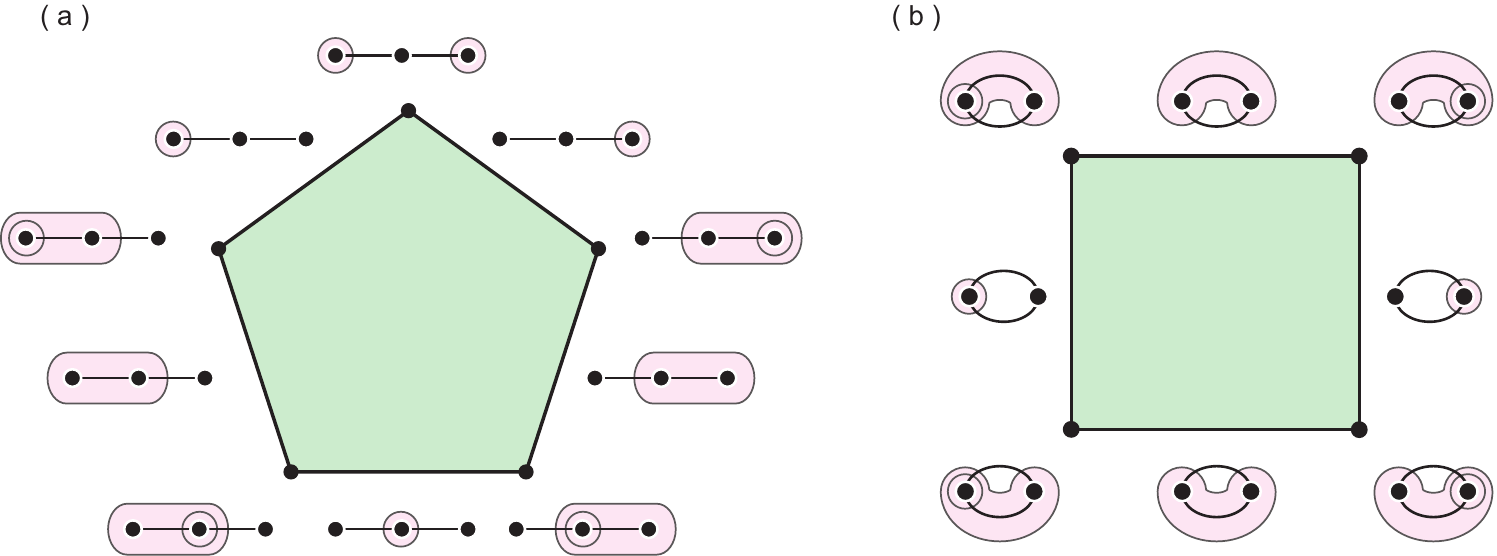}
\caption{Graph associahedra of a path and a multi-edge.}
\label{f:2d-tubes}
\end{figure}
These polytopes were first motivated by De Concini and Procesi in their work on ``wonderful'' compactifications of hyperplane arrangements \cite{dp}.  In particular, if the hyperplane arrangement is associated to a Coxeter system, the graph associahedron $\KG$ appear as tilings of these spaces, where its underlying graph $G$ is the Coxeter graph of the system  \cite{djs}.  These compactified arrangements are themselves natural generalizations of the Deligne-Knudsen-Mumford compactification \M{n} of the real moduli space of curves \cite{dev1}. From a combinatorics viewpoint, graph associahedra arise in numerous areas, ranging from Bergman complexes of oriented matroids to Heegaard Floer homology \cite{blo}.  Most notably, these polytopes have emerged as graphical tests on ordinal data in biological statistics \cite{mps}.

\subsection{}

The combinatorial and geometric structures of these polytopes capture and expose the fundamental concepts of connectivity and nestings, and it is not surprising that there have been several similar notions, such as nested sets \cite{fei}, nested complexes \cite{zel} and the larger class of generalized permutohedra of Postnikov \cite{pos}. However, none of these constructions capture the notion of nested sets of posets, as we do below.  Indeed, our notion of the set of poset tubes is not a classical {building set}, but falls in a  different category altogether.  

In this paper, we construct a new family of convex polytopes which are extensions of nestohedra and graph associahedra via a generalization of building sets. But rather than starting with a set, we begin with a poset $P$.  The resulting \emph{poset associahedron} $\KP$, based on connected lower sets of $P$, cover a wide swath of existing examples from geometric combinatorics, including the permutahedra, associahedra, multiplihedra, graph associahedra, nestohedra, pseudograph associahedra, and their liftings; in fact, all these types are just from two rank posets.  Newly discovered are polytopes capturing associativity information of CW-complexs.

An overview of the paper is as follows:  Section~\ref{s:posets} supplies the definitions of poset associahedra along with several examples, while Section~\ref{s:construct} provides methods of constructing them via induction.  Specialization to nestohedra and permutohedra is given in Section~\ref{s:relation}, and we finish with proofs of the main theorems in Section~\ref{s:proof}.

%
%
\section{Posets} \label{s:posets}
\subsection{}

We begin with some foundational definitions about posets.  The reader is forewarned that definitions here might not exactly match those from earlier works.   A \emph{lower set}  $L$ is a subset of a poset $P$ such that if $y \preceq x \in L$, then $y \in L$.   The \emph{boundary} of an element $x$ is $\partial x :=\{y \in P \suchthat y \prec x\}$.

\begin{defn} 
Let $\bu x := \{y \in P \suchthat \partial y =  \partial x\}$ be the \emph{bundle} of the element $x$. A bundle is \emph{trivial} if $\bu x = \{x\}$.
\end{defn}

Throughout this paper, a poset will be visually represented by its Hasse diagram.  Consider the example of a poset $P$ given on the left side of Figure~\ref{f:bundle}.  The subset $\{1,2,4,5\}$ in part (a), depicted by the highlighted region, is not a lower set since it does not include element $3$.  This poset is partitioned into four bundles,  $\{1,2,3\}$,\, $\{4\}$,\, $\{5\}$,  and\, $\{6, 7, 8\}$, with elements in a bundle having identical boundary.  In particular, notice that all minimal elements of the poset are in one bundle since they share the empty set as boundary.  The following is immediate:

\begin{figure}[h]
\includegraphics{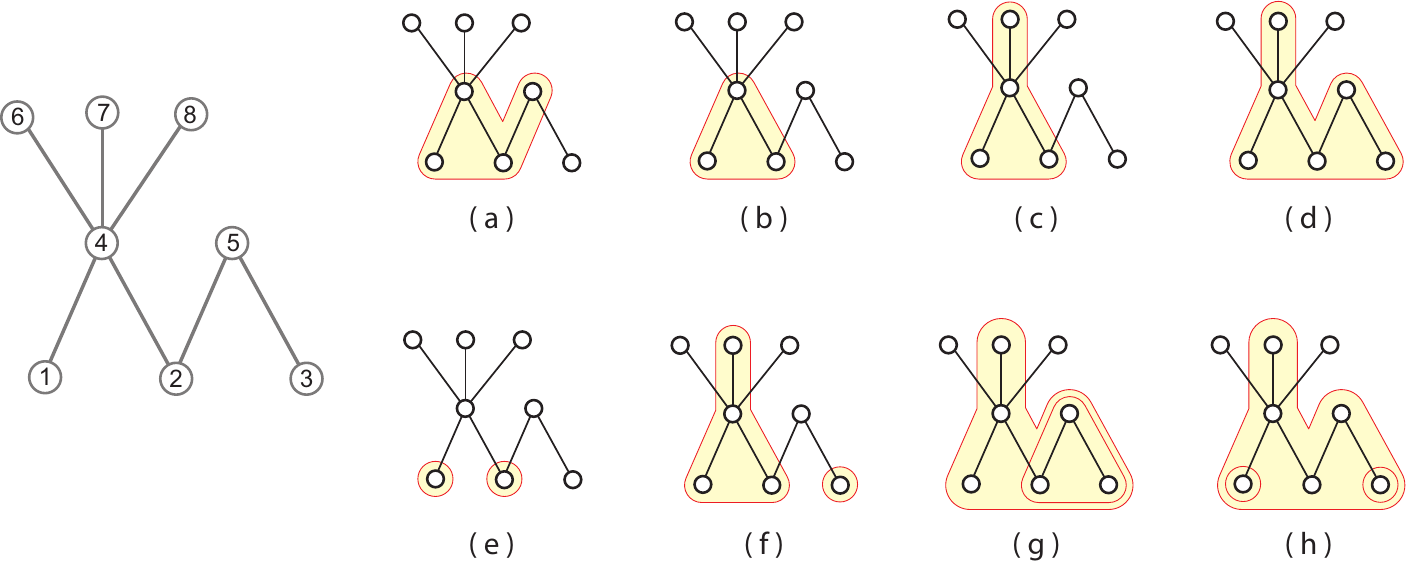}
\caption{Some examples of valid and invalid tubes and tubings.}
\label{f:bundle}
\end{figure}

\begin{lem}
The elements of poset $P$ are partitioned into equivalence classes of bundles.
\end{lem}

\begin{defn}
A lower set is \emph{filled} if, whenever it contains the boundary $\partial x$ of an element $x$, it also intersects the bundle $\bu x$ of that element.
A \emph{tube} is a filled, connected lower set. A \emph{tubing} $T$ is a collection of tubes (not containing all of $P$) which are pairwise disjoint or pairwise nested, and for which the union of every subset of $T$ is filled.
\end{defn}

Figure~\ref{f:bundle}(b) shows the boundary of $\{6, 7, 8\}$, which is an unfilled lower set, whereas (c) is a filled one. Note that parts (c, d) display examples of one tube.  Parts (e, f) display two disjoint tubes which are not  tubings, since the union of the tubes would create an unfilled lower set.  Examples of tubings with two and three components are given by (g, h) respectively.  

\subsection{}

We now present our main result.

\begin{thm} \label{t:combin}
Let $P$ be a poset with $n$ elements partitioned into $b$ bundles.  If $\pi(P)$ is the set of tubings of $P$ ordered by reverse containment, the \emph{poset associahedron} $\KP$ is a convex polytope of dimension $n-b$ whose face poset is isomorphic to $\pi(P)$.
\end{thm}

\noindent This theorem follows from the construction of $\KP$ from truncations, described in Theorem~\ref{t:const} below.  We now pause to illustrate several examples.

\begin{exmp}
Figure~\ref{f:2d-poset} shows the two polytopes of Figure~\ref{f:2d-tubes}, reinterpreted as tubings on posets of their underlying graphs.  Both posets are of \emph{two rank}, the maximum length of any chain of the poset.  Part (a) has 5 elements and 3 bundles, whereas (b) has 4 elements and 2 bundle, both resulting in polygons, as given in Theorem~\ref{t:combin}.
\end{exmp}

\begin{figure}[h]
\includegraphics{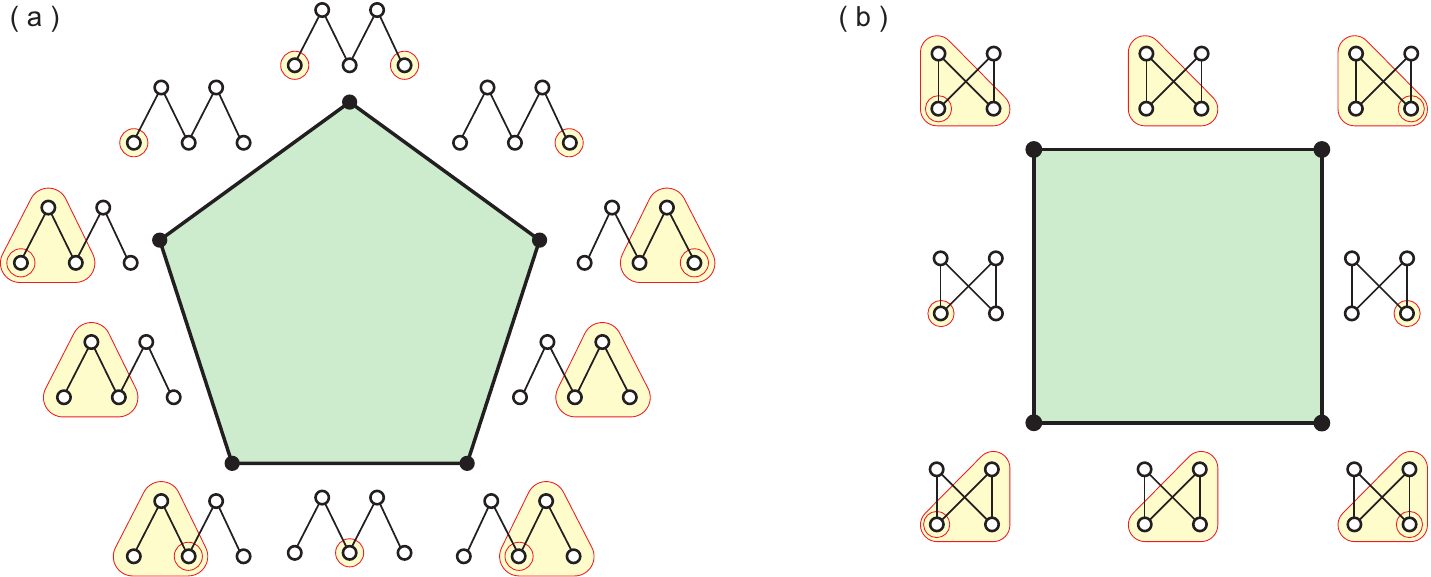} 
\caption{Poset versions of Figure~\ref{f:2d-tubes}.}
\label{f:2d-poset}
\end{figure}

\begin{exmp}
Figure~\ref{f:2d-poset-2} shows two different posets, resulting in identical poset associahedra as Figure~\ref{f:2d-poset}. Part (a) shows a poset structure which does not even come from a CW-complex.  Notice here that the left element of height zero cannot be a tube since it needs to be filled.  Part (b) has a near identical structure to Figure~\ref{f:2d-poset}(b).  Here, the bottom-right element cannot be a tube in itself since it is unfilled.  Because both posets have 5 elements and 3 bundles, the dimension of the polytopes is two.
\end{exmp}

\begin{figure}[h]
\includegraphics{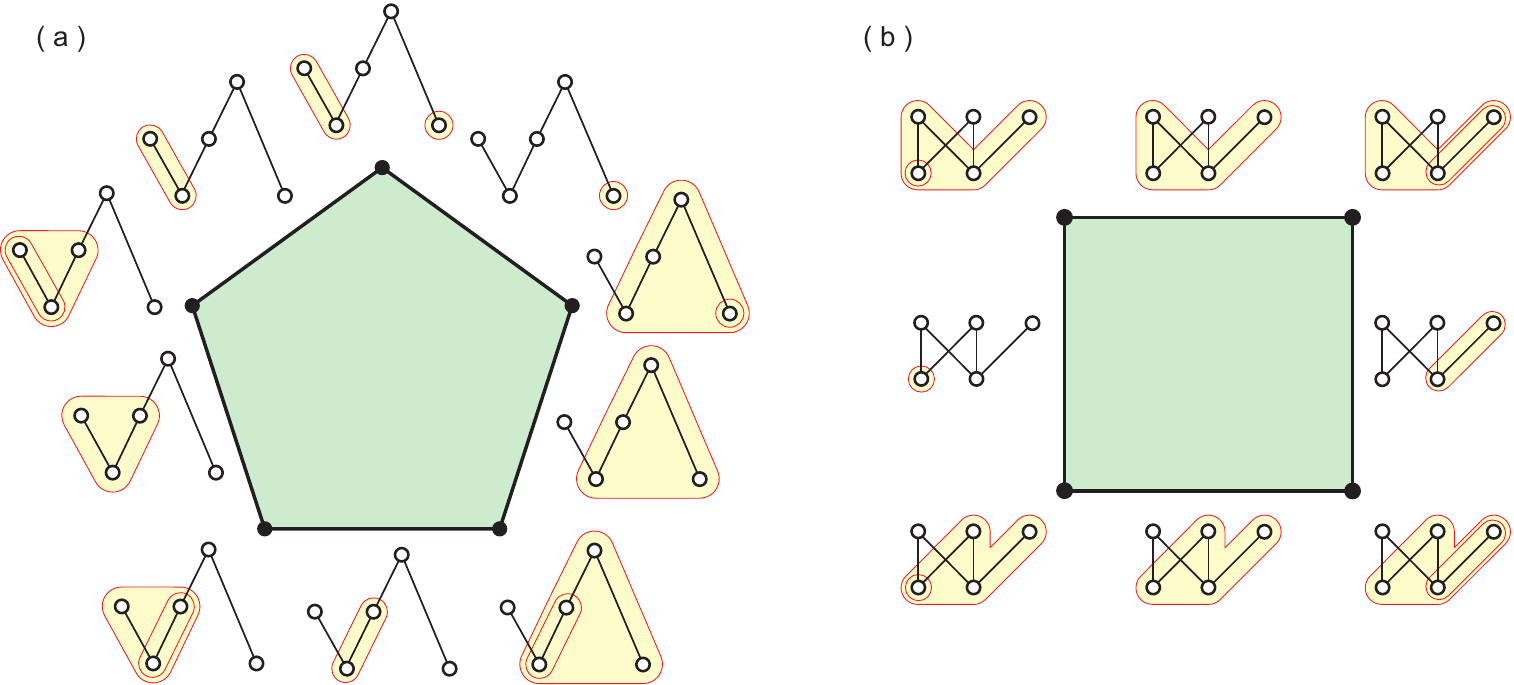} 
\caption{Alternate poset associahedra.}
\label{f:2d-poset-2}
\end{figure}

\begin{exmp}
Three examples of 3D poset associahedra are given in Figure~\ref{f:3d-exmps}.  The cube in (a) can be viewed as an extension of the square in Figure~\ref{f:2d-poset}(b), and  Proposition~\ref{p:cube} generalizes this pattern to the $n$-cube.  A truncation of this cube results in (b), and both posets having 6 elements partitioned into 3 bundles.  Part (c) shows a novel construction of the 3D associahedron $K_5$, with 7 elements partitioned into 4 bundles.
\end{exmp}

\begin{figure}[h]
\includegraphics{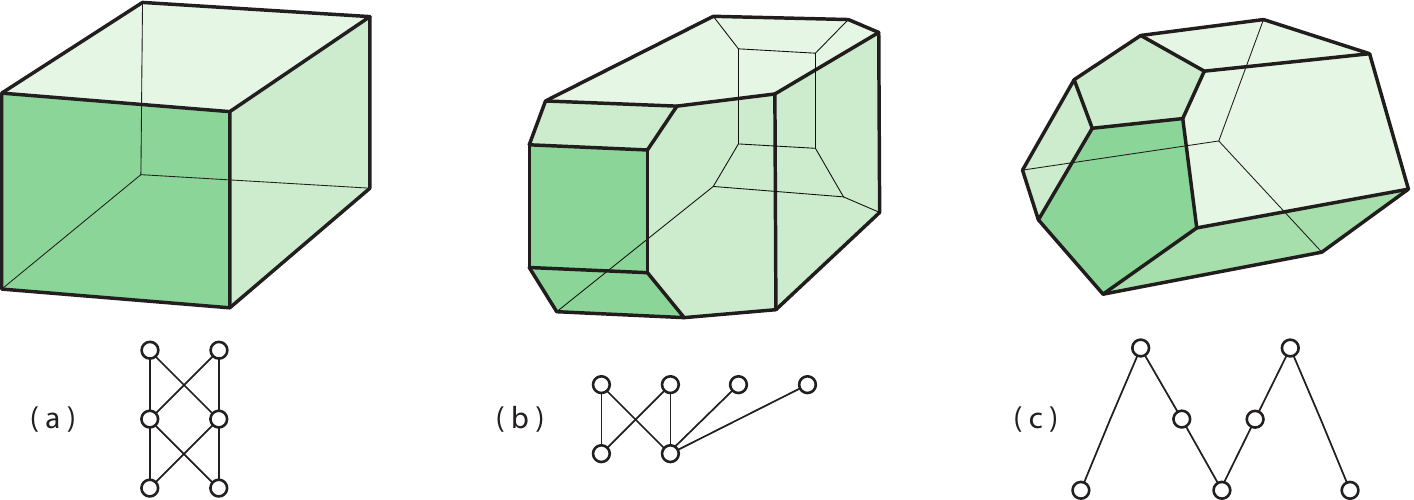}
\caption{Examples of 3D poset associahedra.}
\label{f:3d-exmps}
\end{figure}

\subsection{}

The notion of being ``filled'' appears in different guises:  For graph associahedra, a \emph{tube} of a graph is filled because it is an induced subgraph, catalogued just by listing its set of vertices \cite[Section 2]{cd}.  A \emph{tubing} is filled  because its tubes are ``far apart,''  which is equivalently described by saying that two distinct tubes in a tubing cannot have a single edge connecting them.  Since a simple graph has no bundles, the following is immediate:

\begin{prop}
The graph associahedron $\KG$ can be obtained as a poset associahedron $\KP$, where $P$ is the face poset of graph $G$.
\end{prop}

For pseudographs (having loops and multiple edges), a filled tube is a connected subgraph $t$ where at least one edge between every pair of nodes of $t$ is included if such edges exist \cite[Section 2]{cdf}.  For multiple loops and edges of $G$, the notion of tubes on posets match perfectly with tubes on $G$, and the proposition above extends to the pseudograph associahedron.  The first three examples of Figure~\ref{f:pseudo} displays invalid tubings (all due to not being filled) and the last a valid one; the top row shows tubes on graphs whereas the bottom recasts them on posets.

\begin{figure}[h]
\includegraphics{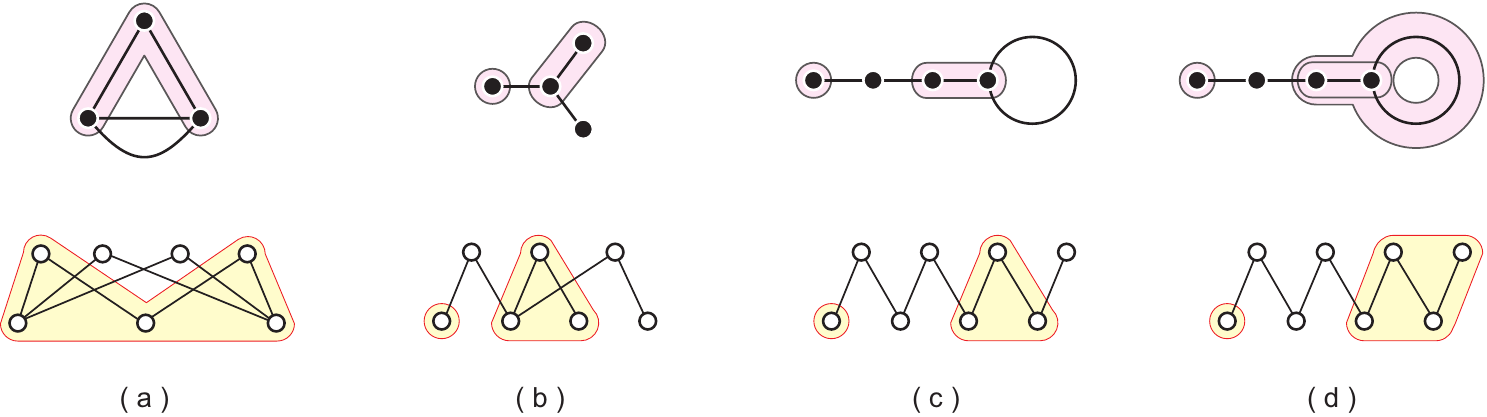}
\caption{Valid and invalid tubings on graphs and posets.}
\label{f:pseudo}
\end{figure}

\begin{rem}
For a single loop attached to a node $v$ of $G$, the pseudograph associahedron defined in \cite{cdf} gave a choice of choosing or ignoring the loop when $v$ is chosen in a tube; see Figure~\ref{f:pseudo}(c) and (d).  In this paper, however, for the sake of consistency in the notion of ``filled'', we always include the loop for poset associahedron.  This allows us to always obtain convex polytopes, rather than unbounded polyhedral chambers of \cite{cdf}.
\end{rem}

The notion of \emph{associativity}, encapsulated by drawing tubes on graphs, has a natural generalizations to higher-dimensional complexes.  In particular, for any CW-complex structure $X$, consider its face poset $P_X$.  The poset associahedron $\KP_X$ captures the analogous information of $X$ that the graph associahedron captures for a graph.

\begin{exmp}
Figure~\ref{f:cello}(a) shows a CW-complex, with three 2-cells, 1-cells, and 0-cells.  Part (b) shows the poset structure of this complex, and (c) its poset associahedron.   
\end{exmp}

\begin{figure}[h]
\includegraphics{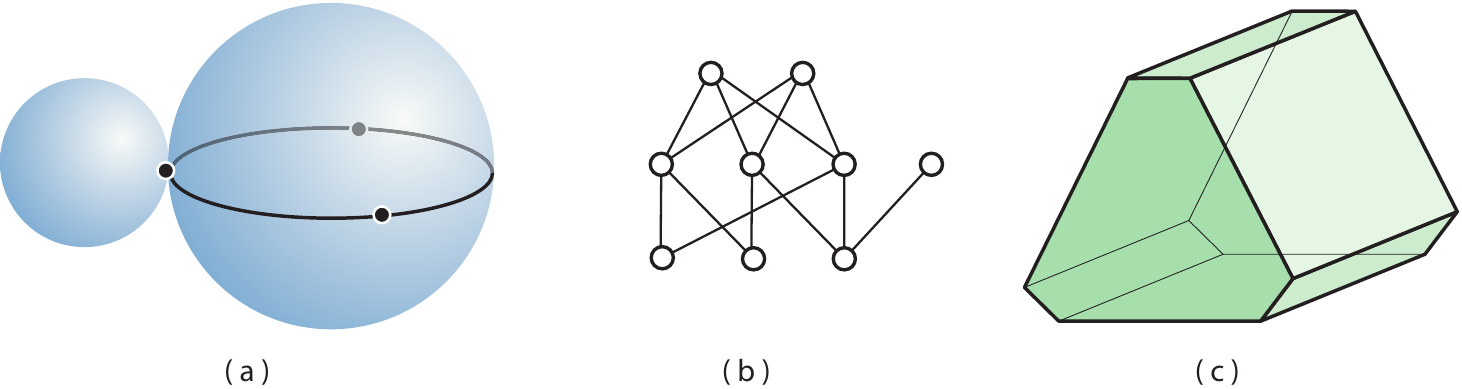}
\caption{The poset associahedron of a CW-complex.}
\label{f:cello}
\end{figure}

%
%
\section{Constructions} \label{s:construct}
\subsection{}

The poset associahedron $\KP$ is recursively built by a series of truncations.  Because the truncation procedure is a delicate one, we present an overview here and save the details for the proof in Section~\ref{s:proof}.  First, the following result allows us to consider only connected posets:

\begin{prop} \label{p:discon}
Let $P$ be a poset with connected Hasse components $P_1$, \ldots, $P_m$.    Then $\KP$ is isomorphic to $\KP_1 \times \cdots \times \KP_m \times \Delta_{m-1}.$
\end{prop}

\begin{proof}
Any tubing of $P$ can be described as:
\begin{enumerate}
\item a listing of tubings $T_1 \in \KP_1, \ \ldots, \ T_k \in \KP_m$, \ and
\item for each component $P_i$ either including or excluding the tube $T_i = P_i$, as long as all tubes $P_i$ are not included. 
\end{enumerate}
The second part of this description is clearly isomorphic to a tubing of the edgeless graph $H_m$ on $m$ nodes.  But from \cite[Section 3]{dev2}, since $\K H_m$ is the simplex $\Delta_{m-1}$, we are done.
\end{proof}

\begin{thm} \label{t:const}
The poset associahedron $\KP$ is constructed inductively on the number of elements of $P$.   Choose a maximal element $x$ of a maximal length chain of $P$.
\begin{enumerate}
\item If bundle $\bu x$ is trivial, truncate  $\K(P  - x)$ to obtain $\KP$.
\item If bundle $\bu x$ is nontrivial, truncate $\K\left( P - (\bu x - x)\right) \times \Delta_{|\bu x - x|}$ to obtain $\KP$.
\end{enumerate}
\end{thm}

\noindent This immediately implies the combinatorial result of Theorem~\ref{t:combin}.   The following is a notable consequence:

\begin{prop} \label{p:trunc}
There are different ways to construct $\KP$, based on the possible choices of maximal elements in the recursive process.
\end{prop}

In certain situations, altering the underlying poset does not affect the polytope.  This occurred in Figures~\ref{f:2d-poset}(b) and~\ref{f:2d-poset-2}(b), and can be presented as

\begin{cor} \label{c:max}
Let $x$ be a maximal element of a maximal chain of $P$ such that $\bu x$ is trivial. If $\partial x$ is connected, then $\KP = \K (P - x)$.
\end{cor}

\begin{proof}
We show that $t$ is a tube of $P$ if and only if $t-x$ is a tube of $P - x$.  If $x \notin t$, then $\partial x \notin t$, and $t$ has the properties of a tube in both $P$ and $P-x$, or in neither. On the other hand, if $x \in t$, then $\partial x \in t$, and so $t$ is connected if and only if $t-x$ is connected. Extending this isomorphism of tubes to tubings preserves this containment.
\end{proof}

\begin{figure}[h]
\includegraphics[width=\textwidth]{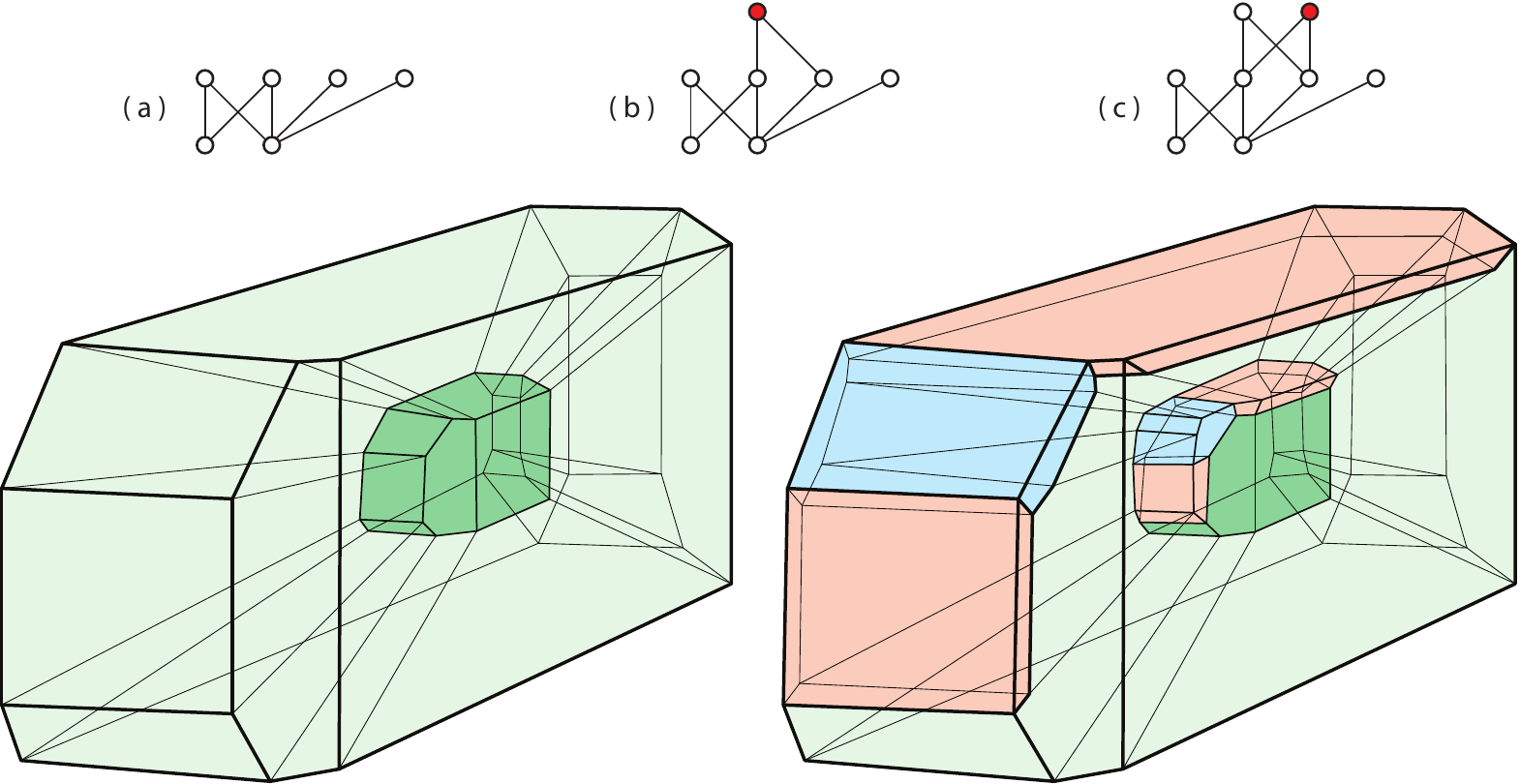}
\caption{Construction of a 4D poset associahedron.}
\label{f:4d}
\end{figure}

\begin{exmp}
A 4D case for Theorem~\ref{t:const} is provided in Schlegel diagram on the right side of Figure~\ref{f:4d}, where the poset steps are drawn above.  Part (a) begins with the poset $P_*$ of Figure~\ref{f:3d-exmps}(b).  In Figure~\ref{f:4d}(b), by Corollary~\ref{c:max}, adding the new maximal element to this poset does not change the structure of the polytope.  Finally, part (c) shows the addition of a nontrivial bundle, now with two elements.  According to Theorem~\ref{t:const}, we first consider the 4D polytope $\KP_* \times \Delta_1$, the left Schlegel diagram of Figure~\ref{f:4d}.  Then truncate certain faces (first the two blue chambers, then the four orange ones) to obtain the 4D poset associahedron drawn on the right.
\end{exmp}

\subsection{}

We close this section with a corollary of Proposition~\ref{p:trunc} as it pertains to the classical associahedron $K_n$.  Interestingly, this construction of the associahedron is novel, though examples of special cases have appeared in different parts of literature, as referenced below.

\begin{prop}
The poset associahedron of the \emph{zigzag} poset with $2n-1$ elements yields the classic associahedron $K_{n+1}$. In particular, the associahedron $K_{n+1}$ is obtained by truncations of codimension two faces of $K_{p+1} \times \Delta_1 \times K_{q+1}$, where $n = p+q$ and $p, q \geq 1$.
\end{prop}

\begin{proof}
The poset $P$ of a path $G$ with $n$ nodes is the zigzag poset $2n-1$ elements; the tubings on $P$ resulting in $\KP$ are in bijection with tubes on the graph $G$.  The enumeration of the different types of truncation comes from removing a maximal element of $P$ and using Theorem~\ref{t:const}(b).
\end{proof}

\begin{rem}
The particular construction of $K_{n+1}$ from $K_n  \times \Delta_1 \times K_2\, \simeq \, K_n  \times \Delta_1 $ appears in another form in the work by Saneblidze and Umble \cite{su} on diagonals of associahedra.
\end{rem}

\begin{figure}[h]
\includegraphics{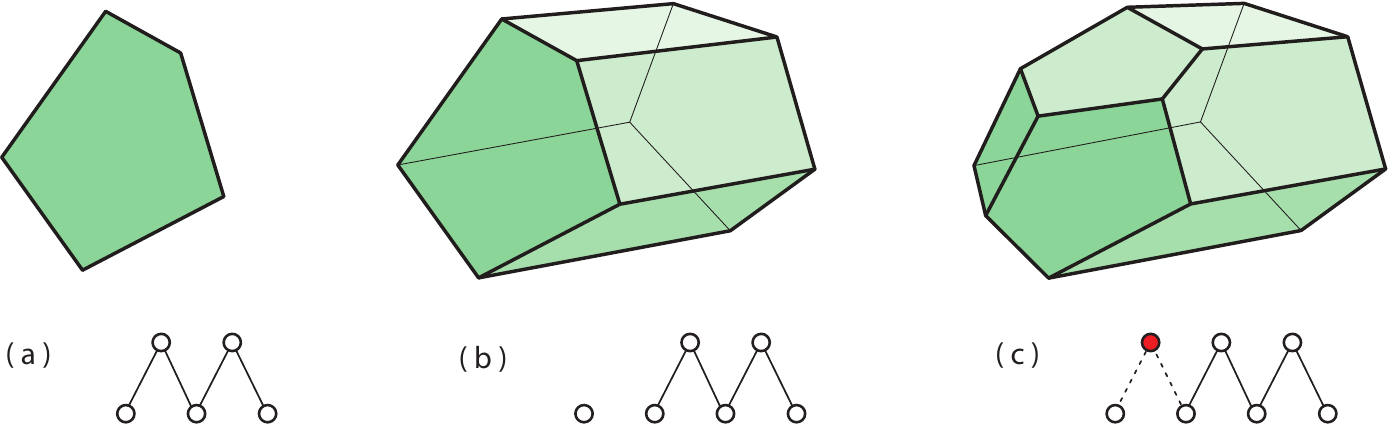}
\caption{The associahedron $K_5$ from a pentagonal prism, $K_4 \times \Delta_1 \times K_2$.}
\label{f:k5-alt1}
\end{figure}

\begin{exmp}
Figure~\ref{f:k5-alt1} considers one construction of the 3D associahedron:  Part (a) begins with the pentagon from Figure~\ref{f:2d-poset}(a), and (b) adds an extra disconnected element.  By Proposition~\ref{p:discon}, the result is the product with $\Delta_1$, a pentagonal prism.  Part (c) connects up the poset with a trivial bundle; two edges of the prism are truncated according to the proof of Theorem~\ref{t:const} to yield the associahedron.  Indeed, Figure~\ref{f:3d-exmps}(c) is formed in an identical manner.
\end{exmp}

\begin{figure}[h]
\includegraphics{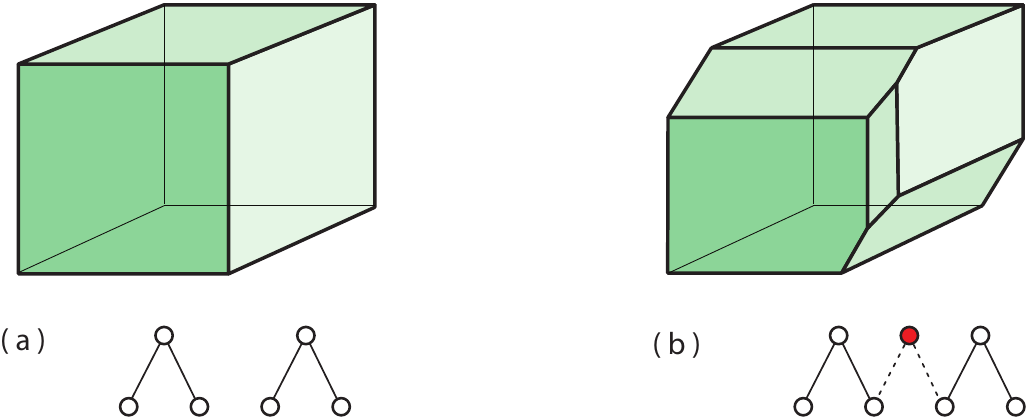}
\caption{The associahedron $K_5$ from a cube, $K_3 \times \Delta_1 \times K_3$.}
\label{f:k5-alt2}
\end{figure}

\begin{exmp}
Figure~\ref{f:k5-alt2} considers another assembly of the 3D associahedron:  Each disconnected component yields an interval, and together (by Proposition~\ref{p:discon}), the result is a cube (a), the product of three intervals.  Part (b) connects up the poset with a trivial bundle, and three edges of its edges are truncated according to the proof of Theorem~\ref{t:const}.  This construction appears in \cite{bv}, motivated by truncating codimension two faces of cubes.
\end{exmp}

\begin{figure}[h]
\includegraphics{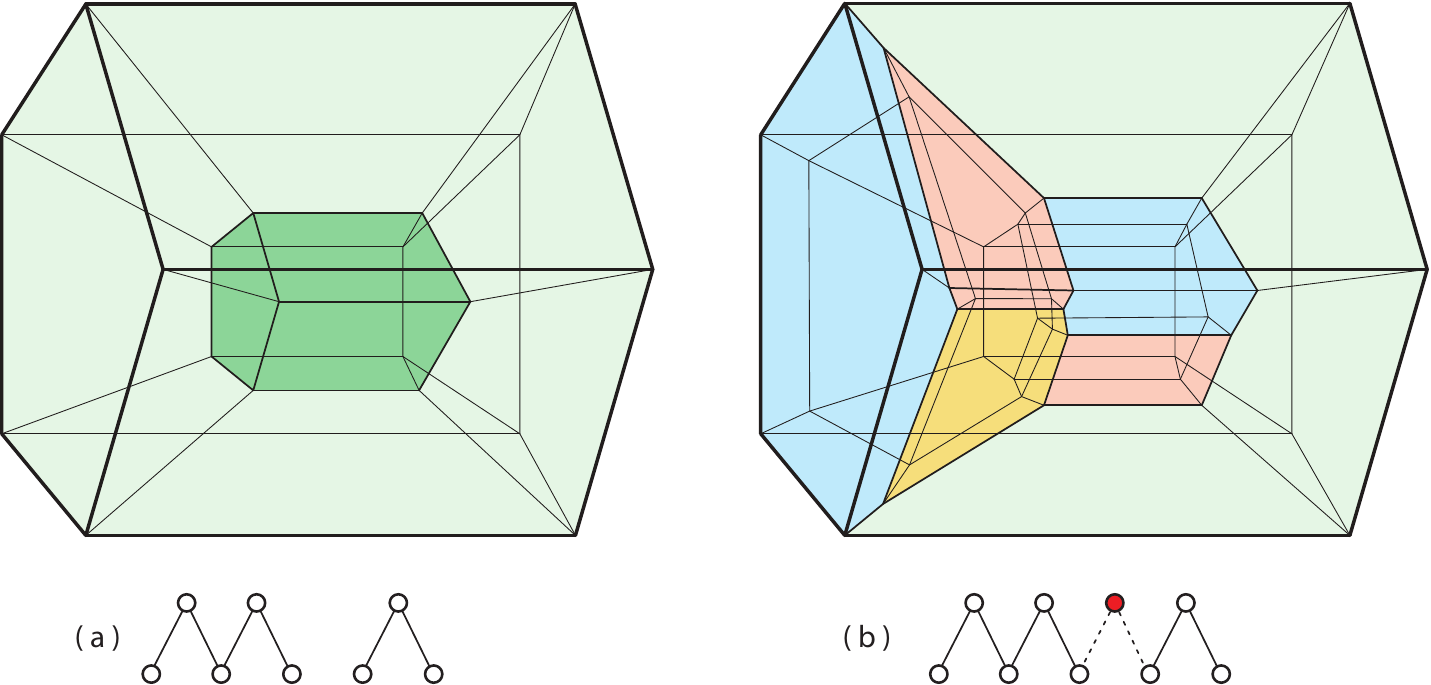}
\caption{The associahedron $K_6$ from $K_4 \times \Delta_1 \times K_3$.}
\label{f:k6-poset}
\end{figure}

\begin{exmp}
In a similar vein, Figure~\ref{f:k6-poset} obtains the 4D associahedron $K_6$ (right side) from truncating five codimension two faces of $K_4 \times \Delta_1 \times K_3$ (left side).  The order of truncation is important:  first the two blue faces, then two orange faces, and finally one yellow face.
\end{exmp}

%
%
\section{Family of Associahedra} \label{s:relation}
\subsection{}

The $(n-1)$-dimensional permutahedron $\per_n$ is the convex hull of the points formed by the action of a finite reflection group on an arbitrary point in Euclidean space. The classic example is the convex hull of all permutations of the coordinates of the Euclidean point $(1, 2, \dots , n)$. Changing edge lengths while preserving their directions results in the \emph{generalized permutohedron}, as defined by Postnikov \cite{pos}. An important subclass of these is the \emph{nestohedra} \cite{zel}: Nestohedra have the feature that each of their faces corresponds to a specific combinatorial set, and the intersection of two faces corresponds to the union of the two sets. For a given set $S$, each nestohedron $N(B)$ is based upon a given building set $B$, whose elements are known as tubes, where $B$ must contain all the singletons of $S$ and must also contain the union of any two tubes whose intersection is nonempty. 

\begin{prop}
\label{p:nesto}
All nestohedra can be obtained as poset associahedra, with posets of two ranks, with no bundles of size greater than one.
\end{prop}

\begin{proof}
Given a building set $B$ of a set $S$, we describe a ranked poset $P_B$, with exactly two ranks, whose tubes are in bijection with $B$ and whose tubings are in bijection with the nested sets of $B$.   The poset $P_B$ has minimal elements given by set $S$, and has maximal elements (each a trivial bundle) given by set $B$, each having boundary exactly the minimal elements that it contains.  The tubes of $P_B$ are all the connected lower sets of $P_B$ generated by a single maximal element of $P_B$ (since each such lower set is automatically filled).  And if a subset $T$  of tubes is not filled, then $T$ must be the boundary of a maximal element in $P_B$, and thus the collection of minimal elements in $T$ make up an element of $B$.
The ordering of nested sets by reverse inclusion corresponds to the ordering of tubings by reverse inclusion, so the nestohedron $N(B)$ is isomorphic to our polytope $\KP_B$. 
\end{proof}

Note that $P_B$ is one of many posets whose polytope is $N(B)$; many more posets (with at most two ranks) can be found.   Start with set $S$ and create any number of new maximal elements, each of which covers some of $S$, where each maximal element is a trivial bundle.  The set of tubes of such a poset $P$ will yield a building set $B$ on the set of minimal elements of $P$, due to the definition of a tube as a connected filled lower set.  This amounts to choosing a subset of the power set of $S$ (a \emph{hypergraph} on $S$), and thus the process of building $\KP$ for such a poset is akin to constructing the hypergraph polytope \cite{hyper}.

\begin{figure}[h]
\includegraphics{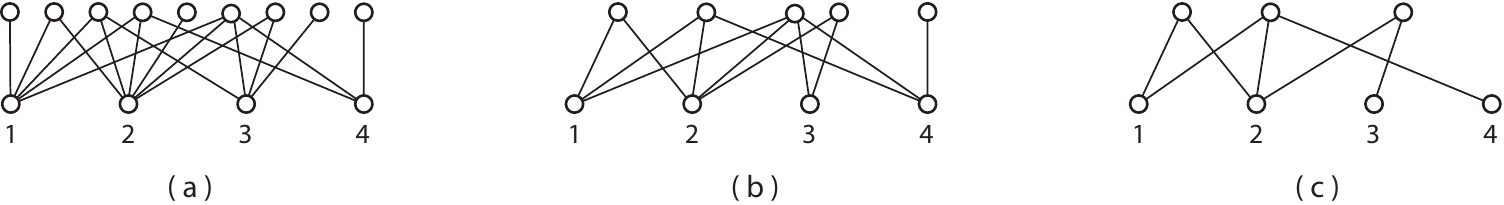}
\caption{Examples of different posets resulting in identical poset associahedra.}
\label{f:nested-many}
\end{figure}

\begin{exmp}
Given set $S = \{1,2,3,4\}$ and building set $B = \{\{1\}, \{2\}, \{3\}, \{4\}, \{12\}, \{23\}, \linebreak \{234\}, \{124\}, S\}$, Figure~\ref{f:nested-many}(a) shows the poset $P_B$ constructed in the proof of Proposition~\ref{p:nesto}.  Moreover, all the posets in this figure result in identical poset associahedra.
\end{exmp}

\subsection{}

Although poset associahedra contain nestohedra, they are a different class than generalized permutohedra.  For instance, Figure~\ref{f:3d-exmps}(b) shows a 3D polytope which has an octagonal face, something not possible for generalized permutohedra.  Figure~\ref{f:octagon} below shows this octagon in detail.  Similarly, the 4D example in  Figure~\ref{f:4d} is not a nestohedron as well.

\begin{figure}[h]
\includegraphics{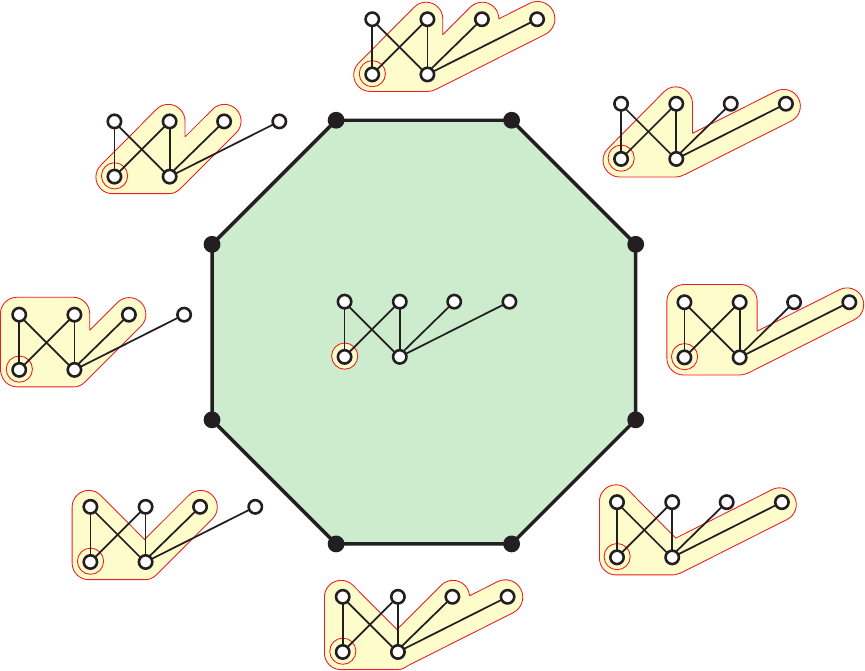}
\caption{Octagonal face of the polyhedron in Figure~\ref{f:3d-exmps}(b).}
\label{f:octagon}
\end{figure}

All graph associahedra and nestohedra are obtained from two rank posets, and it is natural to ask whether all poset associahedra can be obtained from some two rank posets.  For instance, the three rank poset of Figure~\ref{f:2d-poset-2}(a) yields the same $\KP$ as the two rank poset of Figure~\ref{f:2d-poset}(a).  The following shows that higher ranks indeed hold deeper structure.

\begin{thm}
There exists poset associahedra $\KP$ which cannot be found as $\KP'$, for any poset $P'$ of two rank.
\end{thm}

\begin{proof}
Consider the 4D example in Figure~\ref{f:4d}, a three rank poset with $f$-vector $(68, 136, 88, 20)$.   Since this  is not a nestohedron, Proposition~\ref{p:nesto} shows that at least one nontrivial bundle is needed when restricting to two ranks. Using computer calculations, we enumerated the $f$-vectors of all 4D poset associahedra for posets with two ranks and at least one nontrivial bundle.  For each polytope with a matching $f$-vector (about 500 posets), we verified that it was not equivalent to the one in Figure~\ref{f:4d}, with these calculations and comparisons performed using the SAGE package. 
\end{proof}

We close with some special examples.

\begin{prop} \label{p:perm}
Let $\bu x$ be a bundle with $n$ elements of poset $P$ such that all of its elements are maximal.  If $\partial x = P - \bu x$, then $\KP = \K(P - \bu x) \times \per_n$.
\end{prop}

\begin{proof}
A tubing $U \in \KP$ containing no tubes that intersect $\bu x$ can be viewed as a tubing in $\K(P - \bu x)$; call it $\alpha(U)$.  Let $V \in \KP$ be a tubing with only tubes that intersect $\bu x$, where a tube in $V$ is a lower set generated by a subset of $\bu x$, and compatibility of these tubes is equivalent to the subsets being nested.

Let $\Gamma_x$ be the complete graph with $n$ nodes, labeled by elements of $\bu x$.  Let $\beta(U)$ be the tubing on $\Gamma_x$ such that if $t \subset \bu x$ generates a tube in $U$, $t$ is a tube in $\beta (U)$.  Any tubing $T \in \KP$ can be written as a tubing $U$ and a tubing $V$, where the map $T \to (\alpha (U), \beta(V))$ preserves compatibility and is bijective.  The proof follows since the  graph associahedron of a complete graph of $n$ nodes is the permutohedron $\per_n$.  
\end{proof}

\begin{figure}[h]
\includegraphics{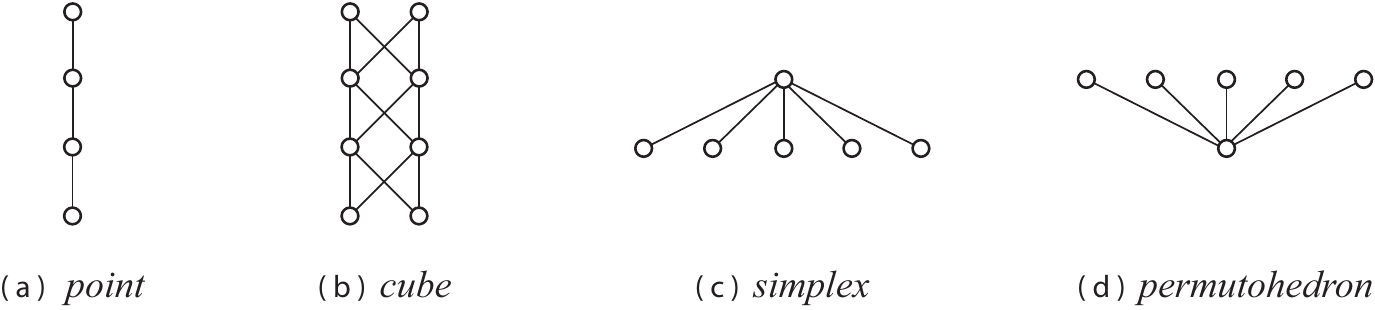}
\caption{Examples of simple posets.}
\label{f:simple}
\end{figure}

\begin{cor} \label{p:cube}
Consider Figure~\ref{f:simple}.  If (a) $P$ is a chain, (b) a cross-stack of $n$ rank, (c) a one rank element with $n+1$ boundary elements, or (d) a bundle with $n$ maximal elements, then $\KP$ is a point, an $n$-cube, an $n$-simplex, or $\per_n$, respectively.
\end{cor}

\begin{proof}
The first follows from Corollary~\ref{c:max}, and the rest from Proposition~\ref{p:perm}.
\end{proof}


%
%
\section{Proofs} \label{s:proof}
\subsection{}

The proof of Theorem~\ref{t:const} is now given, which immediately results in Theorem~\ref{t:combin}.  We proceed by induction on the size of the connected poset $P$:  First, an explicit construction of the polytope $\KP$ is provided, as outlined in Theorem~\ref{t:const}, based on the truncation of a smaller polytope (using the induction hypothesis).  And second, a poset isomorphism is created between the newly constructed $\KP$ and tubings of $P$, establishing the result.  

Throughout this proof, we let $x$ be a maximum element in a maximal chain of $P$.   We first consider the case when $\bu x$ is trivial, and begin with a definition.

\begin{defn}
For trivial $\bu x$, if there exists a pairwise disjoint tubing $T$ of $P-x$ such that 
$$\fil_x(T) \ := \ \{x\} \ \cup \ \{p \in t \suchthat t \in T\}$$ 
is a tube of $P$, call $\fil_x(T)$ the $x$-\emph{fill} of $T$; Figure~\ref{f:filling} gives some examples.
\end{defn}

\begin{figure}[h]
\includegraphics{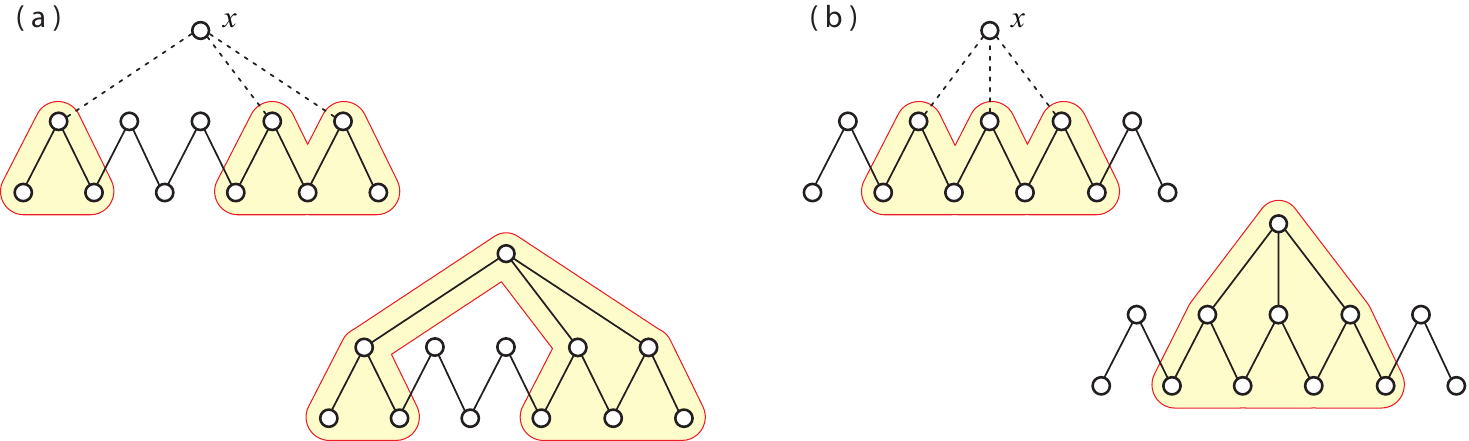}
\caption{Some tubings of $P-x$ (top row) and their $x$-fills (bottom row).}
\label{f:filling}
\end{figure}

\emph{\textcolor{mains}{Truncation algorithm:}} \ \ 
By induction, the theorem holds for $\K(P-x)$; we refer to this object both as a polytope and as the poset of tubings on $P-x$.  Construct $\KP$ by truncating the faces $T$ of $\K(P-x)$ such that $\fil_x(T)$ are tubes of $P$.   The resulting new facets inherit the labelings of $\fil_x(T)$, while retaining the old labels for any facets not truncated; the label of each face is the set of labels of its adjacent facets.   We perform this truncation iteratively, based on \emph{decreasing} size of the tubes $\fil_x(T)$; for a tie, the order of truncation is arbitrary.

\begin{prop}
For trivial $\bu x$, there exists a poset isomorphism $\map$ from $\KP$ to the simple polytope found by truncating $\K(P-x)$ as described above.
\end{prop}

\begin{proof}
It is straightforward that labeling tubes of the facets of our truncated polytope are in bijection with the tubes of $P$.  We show that $\map$ is a bijection of tubings by checking that a collection $T$ of tubes $\{t\}$ of $P$ is a tubing if and only if its corresponding facets $\{f(t)\}$ in $\KP$ intersect at a face.  This simultaneously shows that $\map$ preserves the ordering of tubings.

\emph{\textcolor{mains}{Forward direction:}} \ \ 
We show if $T$ is not a tubing, then $\{f(t) \suchthat t \in T\}$ has empty intersection.  If not a tubing, there is either a subset $S$ of $T$ that is not filled, or there is a pair of nonnested but intersecting tubes in $T$.  First consider the former, when $S$ is unfilled:   If $x$ is the only element whose absence causes $S$ to be unfilled, our truncations will have effectively separated the facets by removing their intersection; otherwise, the conclusion follows by the inductive assumption on $\K(P-x)$.  

Now consider when $T$ contains intersecting, nonnested tubes, say $t_1$ and $t_2$ of $P$.  The faces  $f(t_1)$ and $f(t_2)$ have empty intersection when neither tubes contain $x$ (due to inductive assumption on $\K(P-x)$) or when only one tube $t_1$ contains $x$ (since $f(t_1)$ was formed by truncating a face which was not contained in $f(t_2)$).  When both tubes contain $x$, then there must exist tube $t_*$ of $P$ containing their union.  Now let $t_1=\fil_x(T_1)$ and $t_2=\fil_x(T_2)$ for disjoint tubings $T_1$ and $T_2$ of $P-x$.  Note that we are only concerned if $T_1 \cup T_2$ is a tubing of $P-x$, implying the faces  $f(t_1)$ and $f(t_2)$ originally intersected in $\K(P-x)$.
Let $t_* = \fil_x(T_*)$ for a tubing $T_*$ of disjoint tubes, each containing some of $\partial x$. Since $t_1, t_2 \subset t_*$, the truncation of $T_* \subset T_1 \cup T_2$ will have occurred before either of the latter, and will have involved the truncation of the face labeled by $T_1 \cup T_2,$ ensuring that future facets associated to the truncations of $f(t_1)$ and $f(t_2)$  will not intersect.

\emph{\textcolor{mains}{Backward direction:}} \ \ 
Using finite induction on the truncated tubes in $P$, we show if $T$ is a tubing, then $\{f(t) \suchthat t \in T\}$ has nonempty intersection.  Note that the tubes $\{t_i\}$ in $T$ containing $x$ must be nested $t_1\supset \dots \supset t_m$ since they intersect at least in $x$.  Let $S_i$ be the set of disjoint tubes of $P-x$ such that $t_i = \fil_x(S_i)$.
Thus, associated to a tubing $T$ of $P$, there is a tubing $T_0$ of $P-x$, made of all the tubes of $T$ that do not contain $x$, together with all the tubes in sets $S_i$.

Our argument considers a series of $m$ truncations, proceeding from the list of tubes in $T_0$ to the list of tubes in $T$.  By the construction of $T_0$, restrict attention to the truncations that form facets whose labels are in $T$.  At each step, we show that truncation allows the tubes in the new intermediate list to label facets that intersect.  The base case follows from the induction hypothesis (on the number of elements in our poset) that facets labeled by the tubes of $T_0$ do intersect at a unique face of the simple polytope $\K(P-x)$.

Recursively define an intermediate set of tubes called $T_k$, corresponding to having performed truncations to create the facets labeled by $t_1$ through $t_k$, where
$$
T_{k} \ := \ (T_{k-1} - S_{k}) \ \cup\ (S_k \ \cap \ (T \ \cup \ \bigcup_{{i>k}} S_i)) \ \cup \ (\{t_{k}\}\cap T).
$$
Indeed, after all $m$ truncations, we will have transformed $T_0$ to become $T$, by adding to the list of tubes in $T_0$ all the new tubes $t_i$ in $T$ and subtracting all the tubes which are not in $T$.


By induction, assume that after truncating to create facets $\{t_1, \dots, t_{k-1}\}$, the facets labeled by the tubes in $T_{k-1}$ do indeed intersect.  We use this assumption to show that truncating to create the facet $t_{k} = \fil_x(S_k)$ will preserve the property of intersection: To obtain $T_k$ from $T_{k-1}$,  add $t_k$ but crucially also remove at least one tube; specifically, we claim that at least one tube in $S_k$ will be removed.

Since truncation occurs in decreasing order of containment, the tubes $\{t_{k+1}, \dots, t_m\}$ are all sequentially and properly contained inside of $t_k$.  And since any tube of $T \cap S_k$ must be contained in the smallest of the tubes $t_i$, if there is one or more $t_i \subset t_k$, then the tubes in $S_k$ cannot all be found again among the tubes in the  sets $\{S_{k+1}, \dots, S_m\}$, nor in $T \cap S_k$ itself.  Finally, remove some of $S_k$ when $t_k$ is the smallest tube in $T$ containing $x$, since $T$ (being filled) cannot contain $S_k$.

Therefore, since $T_{k} \cap T_{k-1} = T_{k} - \{t_k\}$ does not contain $S_k$, truncating face $f$ (whose containing facets are labeled by $S_k$) will not separate the facets labeled by the tubes of $T_{k} - \{t_k\}.$ 
However, $f$ does intersect the face where the facets labeled by $T_{k} - \{t_k\}$ intersect, so their intersection will further intersect the new facet labeled by $t_k.$ Thus, the facets labeled by the tubes in $T_k$ will have a nonempty intersection.
\end{proof}

\subsection{}

We now consider the second case in Theorem~\ref{t:combin}, when $\bu x$ is nontrivial.  

\emph{\textcolor{mains}{Truncation algorithm:}} \ \ 
Construct the new polytope $\KP$ by truncating certain faces of 
$$\KP_* \ := \ \K(P-(\bu x - x)) \times \Delta_{|\bu x - x|}.$$ 
Begin by labeling the vertices of $\Delta_{|\bu x - x|}$ with the elements of $\bu x$, and its faces by the subset of vertex labels which they contain.   The faces of $\KP_*$ get labeled by the pairing $(T,B)$, for the corresponding tubing $T$ of $\K(P-(\bu x-x))$ and subset $B$ of $\bu x$.  Now truncate the faces labeled with $(T,B)$ where tubing $T$ has only one tube $t$.   If $x \notin t$, label the resulting new facet with $t$.  After truncations, let the label of each face be the set of labels of its adjacent facets, retaining the old labels for any facets not truncated.  We perform this truncation iteratively, based on \emph{increasing} size of tubes $t$.

\begin{prop}
For nontrivial $\bu x$, there exists a poset isomorphism $\map$ from $\KP$ to the simple polytope found by truncating $\KP_*$ as described above.
\end{prop}

\begin{proof}
It is straightforward that labeling tubes of the facets of our truncated polytope are in bijection with the tubes of $P$.  We show that $\map$ is a bijection of tubings by checking that a collection $T$ of tubes $t$ of $P$ is a tubing if and only if its corresponding facets $f(t)$ in $\KP$ intersect at a face.  This simultaneously shows that $\map$ preserves the ordering of tubings.

\emph{\textcolor{mains}{Forward direction:}}\ \ 
We show if $T$ is not a tubing, then $\{f(t) \suchthat t \in T\}$ has empty intersection.  If not a tubing, there is either a subset $S$ of $T$ that is not filled, or there is a pair of nonnested but intersecting tubes in $T$.  In the case of the former, when $S$ is unfilled, we see implied a further subset $S'$ that was unfilled in the poset $P-(\bu x-x)$, where $S'$ consists of the tubes of $S$ with one modification: Replace any portion of $\bu x$ in those tubes with $x$.  Since, by induction, the facets labeled by tubes of $S'$ have no common intersection in $\K(P-(\bu x-x)),$ and since the product of polytopes preserves this fact, then the faces of the product bearing labels from $S'$ do not have a common intersection.

Now consider when $T$ contains intersecting, nonnested tubes, say $t_1$ and $t_2$ of $P$.  Replace any portion of $\bu x$ contained in them with $x$, resulting in $t'_1 := t_1-(\bu x-x)$ and $t'_2 := t_2-(\bu x-x)$.   If these tubes are still intersecting but nonnested, their facets in $\K(P-(\bu x-x))$ had no intersection, and  this property will be passed along to our new polytope. But if the tubes $t'_1$ and $t'_2$ are nested or equal, then both $t_1$ and $t_2$ contained some of $\bu x$, and we must further consider the intersection 
\begin{equation}
\label{e:inter}
t_1 \ \cap \ t_2 \ \cap \ \bu x \,.
\end{equation}
If this is empty, then $t_1$ and $t_2$ are tubes created by truncating faces of the product polytope $\KP_*$, which in turn corresponded to faces of $\Delta_{|\bu x - x|}$ which did not intersect. Again the non-intersection is inherited by $\K(P-(\bu x-x))$.

Finally if \eqref{e:inter} is nonempty, then it is straightforward to see that the facets labeled by $t_1$ and $t_2$ result from truncating faces that originally do intersect in $\KP_*$. 
Here, there is a third, prior truncation (of a face $f$) of the product polytope that contains the intersection of the faces that are truncated to become $t_1$ and $t_2$.   Indeed, face $f$ gives rise to the facet labeled by the tube $t_1 \cap t_2$, and thus  was labeled in the product by the smaller of $t'_1$ and $t'_2$, paired with \eqref{e:inter}.
Therefore, it is truncated first and effectively separates the others.

\emph{\textcolor{mains}{Backward direction:}}\ \ 
Using finite induction on the truncated tubes $\{t_1, \dots, t_m\}$ in $P$, we show if $T$ is a tubing, then $\{f(t) \suchthat t \in T\}$ has nonempty intersection.  Our argument proceeds by constructing a series of $m$ truncations, showing at each step that the tubes do indeed label facets that intersect.   For a tubing $T$, create the set of pairs 
$$T_0 \ = \ \{\, (t_*,B) \suchthat t_* \ \mbox{is a tube of} \ P-(\bu x-x), \ B \subset \bu x\} \,,$$
where 
\begin{equation*}
t_* \ = \
\begin{cases}
\ (t-(\bu x -x), \ t \cap \bu x) & \hspace{10 pt} \mbox{if}  \ \partial x \subset t\\
\ (t, \bu x)& \hspace{10 pt}  \mbox{otherwise}.
\end{cases}
\end{equation*}
This set gives a list of faces of $\KP_*$, whose intersection is nonempty, providing the base case for truncation:  Since $T$ is a tubing, the tubes of $T$ which contain $\partial x$ are all nested, and by our construction of $T_0$, the set of second elements in the pairs are subsets of $\bu x$ having a common intersection. 

Recursively define an intermediate set of labels $T_k$, formed by performing truncations to create facets labeled by $t_1, \dots, t_k$.  
If $t_k$ is not in $T$, let $T_k$ be $T_{k-1}$.  Otherwise, $T_k$ is formed by discarding from $T_{k-1}$ the pair $(t_k - (\bu x-x), \, t_k \cap \bu x)$ and replacing it with $t_k$ itself.   This corresponds to labeling the new facet with the new tube $t_k$.  We need to show  that faces labeled by  elements of $T_k$ still have a nonempty intersection after the truncation of face $t_k$.
If $t_k$ does not contain $\partial x$, then it labels a facet and its truncation does not change the polytope.  Otherwise,  $t_k$ is either (1) contained in or (2) intersects (but is not nested with) some tube $\{t_{k+1}, \dots, t_m\}$.

In the latter case (2), where $t_k$ intersects such a tube, then $t_k \notin T,$ and we argue that the truncated face does not contain the intersection of the facets labeled by $T_k$.   This follows because $t_k$ either inherits an empty intersection with the faces represented by $T_0$ (from one of the two polytopes in the cross product), or is separated from the intersection of the faces represented by $T_{k-1}$ by an earlier truncation (from the proof in the forward direction).

Now consider the former case (1):  
If $(t_k \cap \bu x) = (t_{k+1} \cap \bu x)$, let $F_*$ be the facet labeled by the pair $(t_k,\bu x)$.  Here, before any truncation, $F_*$ originally contains faces labeled by $(t_k, B)$, for $B\subset \bu x$.  On the other hand, if $(t_k \cap \bu x) \subset (t_{k+1} \cap \bu x)$, let $F_*$ be the facet labeled by the pair $(P-(\bu x-x), \bu x -z)$, for some $z \in (t_{k+1} \cap \bu x) -  (t_k \cap \bu x)$.  In this case, this facet contains faces labeled by $(t,B)$, where $t \subseteq P-(\bu x-x)$ and $B \subseteq \bu x -z$.

In either case, $F_*$ is chosen to contain the truncated face at step $k$ and no other face scheduled to be truncated afterwards.   Moreover, $F_*$ is chosen such that it will eventually be labeled by a tube $t$ that intersects but is not nested with $\{t_{k+1}, \dots, t_m\}$.  Therefore, since we are dealing with simple polytopes, $F_*$ also cannot contain any of the facets represented by the elements of $T_k$, both those labeled by tubes not containing $\partial x$, and those created by earlier truncations.  

Finally, note that the truncated face $t_k$ was assumed to intersect the common intersection of all the other faces represented by $T_{k-1}$, and so the facet created still does. Therefore, the faces represented by $T_k$ have a common intersection, and thus the tubing $T$ will be represented by facets with a common intersection in $\KP$.
\end{proof}

%
%
\bibliographystyle{amsplain}

\end{document}